\documentclass[a4paper,11pt,reqno]{amsart}
\usepackage[utf8]{inputenc}
\usepackage{amsfonts}
\usepackage{amsmath}
\usepackage{amssymb}
\usepackage{amsthm}
\usepackage{verbatim}
\usepackage{color}
\usepackage{hyperref}
\usepackage{epstopdf}
\usepackage{graphicx}

\newtheorem{theorem}{Theorem}[section]

\newtheorem{lemma}[theorem]{Lemma}

\theoremstyle{remark}

\numberwithin{equation}{section}

\renewcommand{\P}{\mathcal{P}_{2g+1}}

\title[The mixed second moment of prime Dirichlet $L$-functions]{The mixed second moment of prime quadratic Dirichlet $L$-functions over function fields}

\author{Christopher G. Best}
\address{Department of Mathematics, University of Exeter, Exeter, EX4 4QF, United Kingdom}
\email{cgb212@exeter.ac.uk}

\date{\today}


\begin{document}

\begin{abstract}
We compute an asymptotic formula for the mixed second moment of the $\mu$-th and $\nu$-th derivatives of quadratic Dirichlet $L$-functions over monic, irreducible polynomials in the function field setting.
\end{abstract}

\maketitle

\section{Introduction}

In this paper we study the mixed second moment of derivatives of quadratic Dirichlet $L$-functions over monic, irreducible polynomials in $\mathbb{F}_q[t]$. In the case of the Riemann zeta function, for integers $\mu,\nu \geq 0$, Ingham \cite{Ingham27} considered a mixed second moment and obtained the asymptotic formula 
\begin{equation*}
    \int_1^T \zeta^{(\mu)}(\tfrac{1}{2}+it) \zeta^{(\nu)}(\tfrac{1}{2}-it) \, dt \sim \frac{1}{\mu+\nu+1} T (\log T)^{\mu+\nu+1},
\end{equation*}
as $T \to \infty$. Since $\zeta^{(\mu)}(\tfrac{1}{2}-it)=\overline{\zeta^{(\mu)}(\tfrac{1}{2}+it)}$, Ingham's result also implies that

\begin{equation*}
    \int_1^T |\zeta^{(\mu)}(\tfrac{1}{2}+it)|^2 \, dt \sim \frac{1}{2\mu+1} T (\log T)^{2\mu+1}.
\end{equation*}
The fourth moment of the derivative of the zeta function was considered by Conrey \cite{Conrey88} who showed that

\begin{equation*}
    \int_1^T |\zeta'(\tfrac{1}{2}+it)|^4 \, dt \sim \frac{61}{1680 \pi^2} T \left( \log\frac{T}{2\pi} \right)^8.
\end{equation*}
Additionally, Conrey proved that as $m \to \infty$,

\begin{equation*}
    \frac{\pi^2}{6} C_{2,m} \sim \frac{1}{16 m^4},
\end{equation*}
where

\begin{equation*}
    C_{k,m}:=\lim_{T \to \infty} T^{-1} \left( \log\frac{T}{2\pi} \right)^{-k^2-2km} \int_1^T |\zeta^{(m)}(\tfrac{1}{2}+it)|^{2k} \, dt.
\end{equation*}

Random matrix theory serves as a very useful tool for formulating conjectures on the analytic properties of the Riemann zeta function and other $L$-functions. In particular, it is now well known that the zeta function can be modelled by the characteristic polynomials of random unitary matrices. In this direction, Conrey, Rubinstein and Snaith \cite{CRS06} studied the moments of the derivative of characteristic polynomials $\Lambda_A(s)$ of matrices $A \in U(N)$ chosen from the Circular Unitary Ensemble (CUE). They proved that as $N \to \infty$ with $k \in \mathbb{N}$,

\begin{equation*}
    \int_{U(N)} |\Lambda_A'(1)|^{2k} \, dA \sim b_k N^{k^2+2k},
\end{equation*}
where $dA$ denotes the normalised Haar measure on the group of unitary matrices $U(N)$ and where $b_k$ has an explicit expression in terms of a determinant involving modified Bessel functions of the first kind. Their result led them to then conjecture that as $T \to \infty$,

\begin{equation*}
    \int_1^T |\zeta'(\tfrac{1}{2}+it)|^{2k} \, dt \sim a_k b_k (\log T)^{k^2+2k},
\end{equation*}
where $a_k$ is the same arithmetic factor appearing in the conjectural formula for the moments of the zeta function itself and is given in terms of an Euler product, see, for example, \cite{CFKRS05,KS00a}. 


Here we consider a mixed second moment of arbitrary derivatives of quadratic Dirichlet $L$-functions in the function field setting. Let $\mathbb{F}_q$ be a finite field with $q$ odd. We denote by $\mathcal{M}$ the set of monic polynomials in $\mathbb{F}_q[t]$ and by $\mathcal{M}_n$ and $\mathcal{M}_{\leq n}$ the sets of monic polynomials of degree $n$ and of degree at most $n$, respectively. Also, $\mathcal{P}$ denotes the set of monic, irreducible  polynomials and $\mathcal{P}_n$ denotes the set of monic, irreducible polynomials of degree $n$. Similarly, $\mathcal{H}$ denotes the set of monic, square-free polynomials and $\mathcal{H}_n$ the set of monic, square-free polynomials of degree $n$.

Our main result concerns the mixed second moment of derivatives of quadratic Dirichlet $L$-functions over monic, irreducible polynomials.

\begin{theorem} \label{joint moment thm}
Let $\mu, \nu \geq 0$ be integers. Then, as $g \to \infty$,

\begin{equation*}
    \frac{1}{|\P|} \sum_{P \in \P} \frac{L^{(\mu)}(\tfrac{1}{2},\chi_P) L^{(\nu)}(\tfrac{1}{2},\chi_P)}{(\log q)^{\mu+\nu}}=c(\mu,\nu) \cdot (2g+1)^{\mu+\nu+3}+O(g^{\mu+\nu+2}),
\end{equation*}
where

\begin{equation}
    c(\mu,\nu)=\frac{1}{2^{\mu+\nu+3} \zeta_q(2)} \left( (-1)^{\mu+\nu} A(\mu,\nu)+\sum_{m=0}^{\mu} \sum_{n=0}^{\nu} \binom{\mu}{m} \binom{\nu}{n} (-2)^{\mu+\nu-m-n} A(m,n) \right),
\end{equation}
$\zeta_q(s)$ is the zeta-function of $\mathbb{F}_q[t]$, and

\begin{equation} \label{A(m,n) def}
    A(m,n):=\frac{1}{2 (m+n+3)} \int_0^1 \left( x^{m+1} (2-x)^n+x^{n+1} (2-x)^m \right) dx.
\end{equation}
\end{theorem}
Our method of proof for Theorem \ref{joint moment thm} involves using an approximate functional equation for the product of two shifted $L$-functions and then taking derivatives of this expression with respect to the shift parameters. We then compute the main term by looking at the contribution of the diagonal terms in the resulting sums and bound the error using the Weil bound. As mentioned in \cite[Remark 1.3]{DD21}, one could derive a formula for the mixed moments by first computing the shifted moments and then taking derivatives. However, this approach leads to very complicated expressions when the orders of the derivatives are not fixed and small.

The study of moments of derivatives of $L$-functions over function fields has already seen a significant amount of interest. Andrade and Yiasemides \cite{AY21} have obtained asymptotic formulae for the first, second and mixed fourth moment of derivatives of Dirichlet $L$-functions, where the average is over all non-trivial characters modulo a monic, irreducible $Q \in \mathbb{F}_q[t]$. Similarly to the Riemann zeta-function, this family of $L$-functions has unitary symmetry type and as such, one sees a very clear analogy between the results of \cite{AY21} and those of Conrey \cite{Conrey88} on the zeta function.

Andrade and Rajagopal \cite{AR16} and Andrade and Jung \cite{AJ21} studied the mean values of the derivatives $L^{(n)}(\tfrac{1}{2},\chi_D)$ over the hyperelliptic ensemble $\mathcal{H}_{2g+1}$. Their general formula \cite[Theorem 3.1]{AJ21} implies that for any integer $n \geq 1$, as $g \to \infty$,

\begin{equation} \label{AJ asymp}
\frac{1}{|\mathcal{H}_{2g+1}|} \sum_{D \in \mathcal{H}_{2g+1}} L^{(n)}(\tfrac{1}{2},\chi_D) \sim \frac{(-1)^n}{2(n+1)} \cdot \mathcal{A}(1) \cdot (2g+1)^{n+1}.
\end{equation}
The factor $\mathcal{A}(1)$ is an arithmetic term given in the form of an Euler product which also appears in the asymptotic formula for the first moment of $L(\tfrac{1}{2},\chi_D)$ of Andrade and Keating \cite{AK12}. In \cite{BJ19}, Bae and Jung used the approach of Florea \cite{Florea17a} to obtain lower order terms in the asymptotic formula for the first moment of $L''(\tfrac{1}{2},\chi_D)$ obtained by Andrade and Rajagopal in \cite{AR16}. It is also shown in \cite{BJ19} that for these quadratic $L$-functions,

\begin{equation*}
    \frac{L'(\tfrac{1}{2},\chi_D)}{-\log q}=g L(\tfrac{1}{2},\chi_D),
\end{equation*}
and so the moments of $L'(\tfrac{1}{2},\chi_D)$ may be obtained easily from the moments of $L(\tfrac{1}{2},\chi_D)$. In particular, one has all the moments of $L'(\tfrac{1}{2},\chi_D)$ up to the fourth using the results of Florea \cite{Florea17a, Florea17b, Florea17c}.

Djankovi{\'c} and \DJ oki{\'c} \cite{DD21} considered the mixed second moment of $L(s,\chi_D)$ and its second derivative and obtained the asymptotic formula \footnote{In \cite{DD21}, the main result is given in terms of the completed $L$-function rather than $L(s,\chi_D)$ but, as discussed in \cite[Remark 1.2]{DD21}, the formula given in (\ref{DD asymp}) follows from \cite[Theorem 1.1]{DD21} and Florea's formula for the second moment of $L(\tfrac{1}{2},\chi_D)$ in \cite{Florea17b}.}

\begin{equation} \label{DD asymp}
\frac{1}{|\mathcal{H}_{2g+1}|} \sum_{D \in \mathcal{H}_{2g+1}} \frac{L(\tfrac{1}{2},\chi_D) L''(\tfrac{1}{2},\chi_D)}{\log^2 q} \sim \frac{1}{80} \cdot \frac{\mathcal{B}(1)}{\zeta_q(2)} \cdot (2g+1)^5.
\end{equation}
Similarly to (\ref{AJ asymp}), $\mathcal{B}(1)$ is an arithmetic factor, given as an Euler product, also appearing in the main term of the second moment of $L(\tfrac{1}{2},\chi_D)$ obtained by Florea \cite{Florea17b}.

The mean values of derivatives of quadratic Dirichlet $L$-functions over monic and irreducible polynomials were first studied by Andrade \cite{And19} who obtained an asymptotic formula for the first moment of $L'(\tfrac{1}{2},\chi_P)$ and $L''(\tfrac{1}{2},\chi_P)$. In \cite{Jung22}, Jung extended the results of \cite{And19} to give an asymptotic formula for the first moment of $L^{(n)}(\tfrac{1}{2},\chi_P)$ over $\P$ for all integers $n \geq 1$. Jung's result implies that at as $g \to \infty$,

\begin{equation} \label{Jung asymp}
    \frac{1}{|\mathcal{P}_{2g+1}|} \sum_{P \in \mathcal{P}_{2g+1}} L^{(n)}(\tfrac{1}{2},\chi_P) \sim \frac{(-1)^n}{2(n+1)} \cdot (2g+1)^{n+1}.
\end{equation}
We note the similarity between (\ref{AJ asymp}) and (\ref{Jung asymp}) as both families of $L$-functions have symplectic symmetry type. We also include the following result on the twisted first moment of $L^{(k)}(\tfrac{1}{2},\chi_P)$ which is a generalisation of Jung's result. Before stating the result, we will denote by $[x]$ the largest integer that is at most $x$ and write $d(f)$ for the degree of a polynomial $f$. For any integers $k,n\geq 0$, we let
\begin{equation*}
    J_{k}(n):=\sum_{m=1}^n m^k.
\end{equation*}
Faulhaber's formula states that 

\begin{equation*}
    J_{k}(n)=\frac{1}{k+1} \sum_{m=0}^{k} \binom{k+1}{m} B_m^{+} n^{k+1-m},
\end{equation*}
where $B_n^{+}$ are the second Bernoulli numbers. In particular, $J_{k}(n)$ is a polynomial in $n$ of degree $k+1$ with zero constant term.

\begin{theorem} \label{twisted first moment}
Let $k \geq 0$ be an integer. Also, let $l \in \mathcal{M}$ and write $l=l_1 l_2^2$ with $l_1,l_2 \in \mathcal{M}$ and $l_1$ square-free. Then, as $g \to \infty$,
    \begin{align*}
        \frac{1}{|\mathcal{P}_{2g+1}|} \sum_{P\in\mathcal{P}_{2g+1}} \frac{L^{(k)}(\tfrac{1}{2},\chi_P) \chi_P(l)}{(\log q)^k} & =\frac{(-1)^k}{|l_1|^{1/2}} \sum_{m=0}^k \binom{k}{m} 2^m d(l_1)^{k-m} J_m([\tfrac{g-d(l_1)}{2}]) \\ 
        & +\frac{1}{|l_1|^{1/2}} \sum_{m=0}^k \binom{k}{m} (-2g)^{k-m} \sum_{i=0}^m \binom{m}{i} 2^i d(l_1)^{m-i} J_i([\tfrac{g-1-d(l_1)}{2}]) \\
        & +O \left (q^{-g/2} g^{k+1} d(l) \right).
    \end{align*}
\end{theorem}

Lastly, random matrix theory provides us with predictions for the asymptotic behaviour of the moments of derivatives of the $L$-functions mentioned above. The families of quadratic $L$-functions $L(s,\chi_D)$ and $L(s,\chi_P)$ are examples of families with symplectic symmetry and so we use the ensemble of random unitary symplectic matrices to model the families and formulate conjectures. In \cite{AB24}, Andrade and the author obtain asymptotic formulae for the joint moments of derivatives of the characteristic polynomials of these matrices. They prove that for non-negative integers $k_1, k_2$ and  $n_1, n_2$,

\begin{equation} \label{Sp asymp}
    \int_{Sp(2N)} \left( \Lambda_A^{(n_1)}(1) \right)^{k_1} \left( \Lambda_A^{(n_2)}(1) \right)^{k_2} dA = b^{Sp}_{k_1, k_2}(n_1, n_2) \cdot (2N)^{k(k+1)/2+k_1 n_2+k_2 n_2} \left( 1+O(N^{-1}) \right),
\end{equation}
where $k=k_1+k_2$. Here, $Sp(2N)$ denotes the group of $2N \times 2N$ unitary symplectic matrices and $dA$ is the Haar measure. Also, the leading order coefficient $b^{Sp}_{k_1, k_2}(n_1, n_2)$ can be written explicitly in the form of a combinatorial sum over partitions (see Theorems 2.1 and 2.2 in \cite{AB24} for a precise expression). The result in (\ref{Sp asymp}) allows for conjectures to made for the corresponding mixed moments of $L$-functions with symplectic symmetry. For instance, for the family $L(s,\chi_P)$, the conjecture is that
\begin{equation*}
    \frac{1}{|\P|} \sum_{P \in \P} \frac{L^{(n_1)}(\tfrac{1}{2},\chi_P)^{k_1} L^{(n_2)}(\tfrac{1}{2},\chi_P)^{k_2}}{(\log q)^{n_1+n_2}} \sim \eta_k \cdot b^{Sp}_{k_1, k_2}(n_1, n_2) \cdot (2g+1)^{k(k+1)/2+k_1 n_1+k_2 n_2},
\end{equation*}
where $\eta_k$ is a certain arithmetic factor in the form of an Euler product. More specifically, $\eta_k$ is the same arithmetic factor present in the conjectural asymptotic formula for the $k$-th moment of $L(\tfrac{1}{2},\chi_P)$ due to Andrade, Jung and Shamesaldeen \cite{AJS21}. See Conjecture 2.2 and Theorem 4.1 in \cite{AJS21} for further details on their conjecture and the arithmetic term. A similar formula is also conjectured to hold for the family $L(s,\chi_D)$ over $\mathcal{H}_{2g+1}$ with a corresponding arithmetic factor. The results of (\ref{AJ asymp}), (\ref{DD asymp}) and (\ref{Jung asymp}) all agree with the prediction of the conjecture based on random matrix theory since it is shown in \cite{AB24} that

\begin{equation*}
    b_{0,1}^{Sp}(0,n)=\frac{(-1)^n}{2(n+1)} \ \textrm{and} \ b_{1,1}^{Sp}(0,2)=\frac{1}{80}.
\end{equation*}

In regards to the mixed second moment considered in Theorem \ref{joint moment thm}, we see that the main term is of the correct size as predicted by the conjecture. The conjecture also states that the leading order coefficient should satisfy

\begin{equation} \label{coeff identity}
    c(n_1,n_2)=\frac{1}{\zeta_q(2)} \cdot b_{1,1}^{Sp}(n_1,n_2),
\end{equation}
since $\eta_2=\zeta_q(2)^{-1}$ is the relevant arithmetic factor for the second moment. In this case, the random matrix theory coefficient $ b_{1,1}^{Sp}(n_1,n_2)$ has the following explicit expression from \cite[Theorem 2.2]{AB24}:

\begin{align*}
    b_{1,1}^{Sp}(n_1,n_2) & =\frac{(-1)^{n_1+n_2}}{2^{n_1+n_2+3}} (n_1!) (n_2!) \sum_{2l_1+2l_2 \leq n_1} \sum_{2m_1+2m_2 \leq n_2} \frac{1}{(n_1-2l_1-2l_2)!} \frac{1}{(n_2-2m_1-2m_2)!} \nonumber \\
    &\qquad \times \frac{(2l_2+2m_2-2l_1-2m_1-2)}{(2l_1+2m_1+3)! (2l_2+2m_2+1)!},
\end{align*}
where the sum is over non-negative integers $l_1, l_2$ and $m_1, m_2$. We do not attempt to prove that (\ref{coeff identity}) holds for all $n_1,n_2$ here but we have checked numerically that it does indeed hold for $n_1,n_2 \leq 20$. 

\section{Background on $L$-functions over function fields}

Here we recall the necessary background on quadratic Dirichlet $L$-functions over function fields. We use \cite{Ros02} as a general reference.

For a polynomial $f \in \mathbb{F}_q[t]$, the norm of $f$ is defined to be $q^{d(f)}$ if $f \neq 0$ and $|f|=0$ for $f=0$.
The zeta function $\zeta_q(s)$ of $\mathbb{F}_q[t]$ is defined for $\mathrm{Re}(s)>1$ by the Dirichlet series and Euler product

\begin{equation*}
    \zeta_q(s)=\sum_{f \in \mathcal{M}} \frac{1}{|f|^s}=\prod_{P \in \mathcal{P}} \left( 1-\frac{1}{|P|^s} \right)^{-1}.
\end{equation*}
As there are $q^n$ monic polynomials of degree $n$, we have that

\begin{equation*}
    \zeta_q(s)=\frac{1}{1-q^{1-s}}.
\end{equation*}
For an integer $n \geq 1$, the Prime Polynomial Theorem states that

\begin{equation*}
    |\mathcal{P}_n|=\frac{q^n}{n}+O \left( \frac{q^{n/2}}{n} \right).
\end{equation*}

Given a monic, irreducible polynomial $P \in \P$, we define the quadratic character $\chi_P$ using the Legendre symbol

\begin{equation*}
    \chi_P(f)=\left( \frac{f}{P} \right).
\end{equation*}
That is,

\begin{align*}
    \chi_P(f)=\begin{cases} 0, & \textup{ if } P|f, \\ 1, & \textup{ if } P \nmid f \textup{ and } f \textup{ is a  square modulo } P, \\ -1, & \textup{ if } P \nmid f \textup{ and } f \textup{ is not a  square modulo } P. \end{cases}
\end{align*}
The quadratic Dirichlet $L$-function attached to the character $\chi_P$ is defined for $\textup{Re}(s)>1$ by 

\begin{equation*}
    L(s,\chi_P)=\sum_{f \in \mathcal{M}} \frac{\chi_P(f)}{|P|^s}=\prod_{Q \in \P} \left( 1-\frac{\chi_P(Q)}{|Q|^s} \right)^{-1}.
\end{equation*}
With the change of variables $u=q^{-s}$, we may write

\begin{equation*}
    \mathcal{L}(u,\chi_P):=L(s,\chi_P)=\sum_{f \in \mathcal{M}} \chi_P(f) u^{d(f)}=\prod_{Q \in \P} \left( 1-\chi_P(Q) u^{d(Q)} \right)^{-1}.
\end{equation*}
We have that $\mathcal{L}(u,\chi_P)$ is in fact a polynomial in $u$ of degree $2g$ and satisfies the functional equation

\begin{equation*}
    \mathcal{L}(u,\chi_P)=(qu^2)^g \mathcal{L} \left( \frac{1}{qu},\chi_P \right).
\end{equation*}
By the Riemann Hypothesis for curves over finite fields, proven by Weil \cite{Weil48}, all of the zeros of $\mathcal{L}(u,\chi_P)$ lie on the circle $|u|=q^{-1/2}$.
 
Now, we denote the divisor function on $\mathbb{F}_q[t]$ by $\tau(f)$ which satisfies

\begin{equation*}
    \sum_{f \in \mathcal{M}_n} \tau(f)=(n+1) q^n,
\end{equation*}
and for $\alpha,\beta \in \mathbb{C}$, we let

\begin{equation*}
    \tau_{\alpha,\beta}(f)=\sum_{f=f_1 f_2} \frac{1}{|f_1|^{\alpha} |f_2|^{\beta}}.
\end{equation*}
Also, for integers $m,n \geq 0$ we will denote

\begin{equation*}
    \tau^{(m,n)}(f):=\frac{\partial^{m+n}}{\partial\alpha^m \partial\beta^n} \tau_{\alpha,\beta}(f) |_{\alpha=\beta=0}=(-\log q)^{m+n} \sum_{f=f_1 f_2} d(f_1)^m d(f_2)^n,
\end{equation*}
and note that we have the bound

\begin{equation*}
    |\tau^{(m,n)}(f)| \ll \sum_{f=f_1 f_2} d(f)^{m+n} \ll \tau(f) d(f)^{m+n}.
\end{equation*}

Lastly, we have the following Weil bound for character sums over monic, irreducible polynomials.

\begin{lemma} \label{Weil bound}
    For $f \in \mathcal{M}$ not a square, as $g \to \infty$,

    \begin{equation*}
        \frac{1}{|\P|} \sum_{P \in \P} \chi_P(f) \ll q^{-g} d(f).
    \end{equation*}
\end{lemma}

\begin{proof}
This follows from equation (2.5) in \cite{Rud10} and the Prime Polynomial Theorem.
\end{proof}

\section{The twisted first moment of the $k$-th derivative}

Here we will prove Theorem \ref{twisted first moment} following the approach used in \cite{Jung22}. Let $l \in \mathcal{M}$ and write $l=l_1 l_2^2$ with $l_1, l_2 \in \mathcal{M}$ and $l_1$ square-free. For $h \in \{g,g-1\}$, we define the sum

\begin{align*}
    S_h (m;l) & :=\sum_{f \in \mathcal{M}_{\leq h}} \frac{d(f)^m}{|f|^{1/2}} \frac{1}{|\mathcal{P}_{2g+1}|} \sum_{P\in\mathcal{P}_{2g+1}} \chi_P(fl) \\
    & =\sum_{n=0}^h n^m q^{-n/2} \sum_{f\in\mathcal{M}_n} \frac{1}{|\mathcal{P}_{2g+1}|} \sum_{P\in\mathcal{P}_{2g+1}} \chi_P(fl),
\end{align*}
and compute an asymptotic formula for $S_h (m;l)$ in the following lemma.

\begin{lemma} \label{S asymptotic} 
For an integer $m \geq 0$ and $h \in \{g,g-1\}$, we have that as $g \to \infty$,

    \begin{equation*}
        S_h (m;l)=\frac{1}{|l_1|^{1/2}} \sum_{i=0}^m \binom{m}{i} 2^i d(l_1)^{m-i} J_i([\tfrac{h-d(l_1)}{2}])+O \left( q^{-g/2} g^{m+1} d(l) \right).
    \end{equation*}
\end{lemma}

\begin{proof}
We split the sum $S_h (m;l)=S_h (m;l)_{\square}+S_h (m;l)_{\neq\square}$ according to whether $fl=\square$ or $fl\neq\square$. For the contribution of non-squares, we use Lemma \ref{Weil bound} to obtain

\begin{align*}
    |S_h (m;l)_{\neq\square}| & \ll \sum_{n=0}^g n^m q^{-n/2} \sum_{\substack{f\in\mathcal{M}_n \\ fl\neq\square}} \left| \frac{1}{|\P|} \sum_{P \in \P} \chi_P(fl) \right| \\
    & \ll q^{-g} \sum_{n=0}^g n^m q^{-n/2} \sum_{f\in\mathcal{M}_n} d(fl) \\
    & \ll q^{-g} (g+d(l)) \sum_{n=0}^g n^m q^{n/2} \\
    & \ll q^{-g/2} g^m (g+d(l)).
\end{align*}

For the contribution of the squares in $S_h (m;l)_{\square}$, we use the facts that $\chi_P(f^2)=\chi_P(f)^2$ and since $d(f) \leq g$, we have that $P \nmid f$ for all $P \in \P$. Thus, for $fl=\square$, we have

\begin{equation*}
    \sum_{P \in \P} \chi_P(fl)=\sum_{P \in \P} 1-\sum_{\substack{P \in \P \\ P|l}} 1=|\P|+O(d(l)).
\end{equation*}
Now, as $fl=\square$, we write $f=l_1 f_1^2$ with $f_2$ monic. Then, since $d(f)=d(l_1)+2d(f_1) \leq h$, we can rewrite the sum over $f_1$ with $d(f_1) \leq (h-d(l_1))/2$ which gives us that

\begin{align*}
    S_h (m;l)_{\square} & =\sum_{n=0}^h n^m q^{-n/2} \sum_{\substack{f\in\mathcal{M}_n \\ fl=\square}} \frac{1}{|\mathcal{P}_{2g+1}|} \sum_{P \in \P} \chi_P(fl) \\
    & =\sum_{n=0}^h n^m q^{-n/2} \sum_{\substack{f \in \mathcal{M}_n \\ fl=\square}} 1+O\left( \frac{d(l)}{|\P|} \sum_{n=0}^h n^m q^{-n/2} \sum_{\substack{f\in\mathcal{M}_n \\ fl=\square}} 1 \right) \\
    & =q^{-d(l_1)/2} \sum_{n=0}^{(h-d(l_1))/2} (d(l_1)+2n)^m q^{-n} \sum_{f_1 \in \mathcal{M}_n} 1+O\left( q^{-2g} g \, d(l) \sum_{n=0}^h n^m q^{n/2} \right) \\
    & =\frac{1}{|l_1|^{1/2}} \sum_{n=0}^{(h-d(l_1))/2} (d(l_1)+2n)^m+O\left( q^{-3g/2} g^{m+1} d(l) \right),
\end{align*}
where we have used the Prime Polynomial Theorem in bounding the error. For the main term, we use a binomial expansion and the function $J_i(n)=\sum_{m=0}^n m^i$ to write it as

\begin{align*}
    \frac{1}{|l_1|^{1/2}} \sum_{n=0}^{(h-d(l_1))/2} (2n+d(l_1))^m & =\frac{1}{|l_1|^{1/2}} \sum_{n=0}^{(h-d(l_1))/2} \sum_{i=0}^m \binom{m}{i} (2n)^i d(l_1)^{m-i} \\
    & =\frac{1}{|l_1|^{1/2}} \sum_{i=0}^m \binom{m}{i} 2^i d(l_1)^{m-i} \sum_{n=0}^{(h-d(l_1))/2} n^i \\
    & =\frac{1}{|l_1|^{1/2}} \sum_{i=0}^m \binom{m}{i} 2^i d(l_1)^{m-i} J_i([\tfrac{h-d(l_1)}{2}]).
\end{align*}
Combining this with the bounds for the error terms completes the proof.
\end{proof}

\begin{proof}[Proof of Theorem \ref{twisted first moment}]

Using the expression for $L^{(k)}(\tfrac{1}{2},\chi_P)$ given in \cite[Lemma 5.1]{AJ21} and multiplying by $\chi_P(l)$, we have that

\begin{equation*}
    \frac{L^{(k)}(\tfrac{1}{2},\chi_P) \chi_P(l)}{(\log q)^k}=(-1)^k \sum_{n=0}^g n^k q^{-n/2} \sum_{f\in\mathcal{M}_n} \chi_P(fl)+\sum_{m=0}^k \binom{k}{m} (-2g)^{k-m} \sum_{n=0}^{g-1} n^m q^{-n/2} \sum_{f\in\mathcal{M}_n} \chi_P(fl).
\end{equation*}
Taking the average over $\mathcal{P}_{2g+1}$, we can then write

\begin{equation*}
    \frac{1}{|\P|} \sum_{P \in \P} \frac{L^{(k)}(\tfrac{1}{2},\chi_P) \chi_P(l)}{(\log q)^k}=(-1)^k S_g (k;l)+\sum_{m=0}^k \binom{k}{m} (-2g)^{k-m} S_{g-1} (m;l).
\end{equation*}
Using this expression for the twisted moment and applying the formula for $S_h (m;l)$ in Lemma \ref{S asymptotic} completes the proof.
    
\end{proof}

\section{Joint moments of the $\mu$-th and $\nu$-th derivatives}

In this section will prove Theorem \ref{joint moment thm}. First, for integers $m,n \geq 0$ and $h \in \{2g,2g-1\}$, let

\begin{equation*}
    T_h(m,n)=\sum_{f\in\mathcal{M}_{\leq h}} \frac{\tau^{(m,n)}(f)}{|f|^{1/2}} \frac{1}{|\P|} \sum_{P\in\P} \chi_P(f).
\end{equation*}
The following lemma establishes the asymptotic behaviour of the sums $T_h(m,n)$ as $g \to \infty$. 

\begin{lemma} \label{T asymptotic}
    For integers $m,n \geq 0$ and $h \in \{2g, 2g-1\}$, we have that as $g \to \infty$,

    \begin{equation}
        T_h(m,n)=\frac{(-\log q)^{m+n} A(m,n)}{\zeta_q(2)} g^{m+n+3}+O(g^{m+n+2}),
    \end{equation}
    with $A(m,n)$ as defined in (\ref{A(m,n) def}).
\end{lemma}

\begin{proof}
We write

\begin{equation*}
     T_h(m,n)=T_h(m,n)_{\square}+T_h(m,n)_{\neq \square},
\end{equation*}
where $T_h(m,n)_{\square}$ and $T_h(m,n)_{\neq \square}$ denote the sums over $f$ a perfect square and $f$ not a square, respectively. For the terms with $f \neq \square$, we bound the sum using Lemma \ref{Weil bound} to obtain

\begin{align*}
    |T_h(m,n)_{\neq \square}| & \ll \sum_{\substack{f\in\mathcal{M}_{\leq 2g}\\ f\neq\square}} \frac{\tau^{(m,n)}(f)}{|f|^{1/2}} \left| \frac{1}{|\P|} \sum_{P\in\P} \chi_P(f) \right| \\ 
    & \ll q^{-g} \sum_{f\in\mathcal{M}_{\leq 2g}} \frac{\tau^{(m,n)}(f)}{|f|^{1/2}} d(f) \\
    & \ll g q^{-g} \sum_{j=0}^{2g} q^{-j/2} \sum_{f\in\mathcal{M}_j} \tau(f) d(f)^{m+n} \\
    & \ll g q^{-g} \sum_{j=0}^{2g} j^{m+n+1} q^{j/2} \\
    & \ll g^{m+n+2},
\end{align*}
where we have used the fact that $\sum_{f\in\mathcal{M}_j} \tau(f)\ll j q^j$. For the terms in $T_h(m,n)_{\square}$ with $f=\square$, we have

\begin{equation*}
    \frac{1}{|\P|} \sum_{P\in\P} \chi_P(f)=1,
\end{equation*}
since $d(f) \leq 2g$ and so $P \nmid f$ for all $P \in \P$. Thus, we have that

\begin{equation*}
    T_h(m,n)_{\square}=\sum_{\substack{f\in\mathcal{M}_{\leq h}\\ f=\square}} \frac{\tau^{(m,n)}(f)}{|f|^{1/2}}=\sum_{f\in\mathcal{M}_{\leq h/2}} \frac{\tau^{(m,n)}(f^2)}{|f|}.
\end{equation*}
We next write the main term $T_h(m,n)_{\square}$ as

\begin{equation*}
    \sum_{f\in\mathcal{M}_{\leq h/2}} \frac{\tau^{(m,n)}(f^2)}{|f|}=(-\log q)^{m+n} \sum_{j=0}^{h/2} q^{-j} \sum_{f \in \mathcal{M}_j} \sum_{f^2=f_1 f_2} d(f_1)^m d(f_2)^n.
\end{equation*}
Then, for a given $j$, we use the hyperbola method and write

\begin{align} \label{hyperbola}
    \sum_{f \in \mathcal{M}_j} \sum_{f^2=f_1 f_2} d(f_1)^m d(f_2)^n & =\sum_{f \in \mathcal{M}_j} \sum_{\substack{f^2=f_1 f_2 \\ d(f_1)=j}} j^{m+n}+\sum_{f \in \mathcal{M}_j} \sum_{\substack{f^2=f_1 f_2 \\ d(f_1)<j}} d(f_1)^m d(f_2)^n \nonumber \\
    &\qquad +\sum_{f \in \mathcal{M}_j} \sum_{\substack{f^2=f_1 f_2 \\ d(f_2)<j}} d(f_1)^m d(f_2)^n.
\end{align}
The second and third sums here are similar so we need only focus on the first two. In particular, it suffices to evaluate the sum

\begin{equation} \label{sum 1}
    \sum_{f \in \mathcal{M}_j} \sum_{\substack{f^2=f_1 f_2 \\ d(f_1)=k}} d(f_1)^m d(f_2)^n=k^m (2j-k)^n \sum_{f \in \mathcal{M}_j} \sum_{\substack{f^2=f_1 f_2 \\ d(f_1)=k}} 1,
\end{equation}
for $k \leq j$. To compute the above sum, we observe that $f_1 f_2=\square$ if and only if $f_1=l_1 l_2^2$ and $f_2=l_1 l_3^2$ with $l_1, l_2, l_3 \in \mathcal{M}$ and $l_1$ square-free. Also, since $d(f_1)=k$, we must have $d(l_1) \leq k$ and $d(l_2)=(k-d(l_1))/2$. Then, as $d(f^2)=d(f_1)+d(f_2)=2j$, we have that $d(l_3)=j-k/2-d(l_1)/2 \geq 0$. So, by summing over $l_1, l_2$ and $l_3$, we have that

\begin{align*}
    \sum_{f \in \mathcal{M}_j} \sum_{\substack{f^2=f_1 f_2 \\ d(f_1)=k}} 1 & =\sum_{\substack{l_1 \in \mathcal{H} \\ d(l_1) \leq k \\ k-d(l_1) \ \textup{even}}} \sum_{\substack{l_2 \in \mathcal{M} \\ d(l_2)=(k-d(l_1))/2}} \sum_{\substack{l_3 \in \mathcal{M} \\ d(l_3)=j-k/2-d(l_1)/2}} 1 \\
    & =q^{j-k/2} \sum_{\substack{l_1 \in \mathcal{H} \\ d(l_1) \leq k \\ k-d(l_1) \ \textup{even}}} q^{-d(l_1)/2} \sum_{\substack{l_2 \in \mathcal{M} \\ d(l_2)=(k-d(l_1))/2}} 1 \\
    & =q^j \sum_{\substack{l_1 \in \mathcal{H} \\ d(l_1) \leq k \\ k-d(l_1) \ \textrm{even}}} q^{-d(l_1)}.
\end{align*}
For the final sum over $l_1$, if $k$ is even, then

\begin{equation*}
    \sum_{\substack{l_1 \in \mathcal{H} \\ d(l_1) \leq k \\ k-d(l_1) \ \textup{even}}} q^{-d(l_1)}=\sum_{\substack{i=0 \\ i \ \textup{even}}}^k \sum_{l_1 \in \mathcal{H}_i} q^{-i}=1+\sum_{i=1}^{k/2} \sum_{l_1 \in \mathcal{H}_{2i}} q^{-2i}=1+(1-q^{-1}) \frac{k}{2},
\end{equation*}
where we have used the fact that for $n \geq 1$,

\begin{equation*}
    |\mathcal{H}_n|=\frac{q^n}{\zeta_q(2)}=q^n (1-q^{-1}).
\end{equation*}
Similarly, if $k$ is odd, we have

\begin{equation*}
    \sum_{\substack{l_1 \in \mathcal{H} \\ d(l_1) \leq k \\ k-d(l_1) \ \textup{even}}} q^{-d(l_1)}=1+(1-q^{-1}) \frac{k-1}{2}.
\end{equation*}
Thus, we have that

\begin{equation*}
    \sum_{f \in \mathcal{M}_j} \sum_{\substack{f^2=f_1 f_2 \\ d(f_1)=k}} 1=q^j \left( 1+(1-q^{-1}) [k/2] \right),
\end{equation*}
and incorporating this into (\ref{hyperbola}) and (\ref{sum 1}) gives us

\begin{align} \label{sum 2}
    & \sum_{f \in \mathcal{M}_j} \sum_{f^2=f_1 f_2} d(f_1)^m d(f_2)^n \nonumber \\
    & =q^j \left( j^{m+n} \left( 1+(1-q^{-1}) [j/2] \right)+\sum_{k=0}^{j-1} \left( k^m (2j-k)^n+k^n (2j-k)^m \right) \left( 1+(1-q^{-1}) [k/2] \right) \right).
\end{align}

We now approximate the expression in (\ref{sum 2}) as $j \to \infty$. As

\begin{equation*}
    \left( 1+(1-q^{-1}) [k/2] \right)=(1-q^{-1}) \frac{k}{2}+O(1),
\end{equation*}
the first term in the brackets in (\ref{sum 2}) is

\begin{equation*}
    (1-q^{-1}) \frac{j^{m+n+1}}{2}+O(j^{m+n}).
\end{equation*}
Next, we have that

\begin{equation} \label{sum 3}
    \sum_{k=0}^{j-1} k^m (2j-k)^n \left( 1+(1-q^{-1}) [k/2] \right)=\frac{(1-q^{-1})}{2} \sum_{k=0}^{j-1} k^{m+1} (2j-k)^n+O \left( \sum_{k=0}^{j-1} k^m (2j-k)^n \right),
\end{equation}
so we now focus on the sum $\sum_{k=0}^{j-1} k^m (2j-k)^n$. On the one hand, by writing it in terms of a Riemann sum, we have

\begin{equation*}
    \sum_{k=0}^{j-1} k^m (2j-k)^n=j^{m+n+1} \cdot \frac{1}{j} \sum_{k=0}^{j-1} \left( \frac{k}{j} \right)^m \left( 2-\frac{k}{j} \right)^n \sim j^{m+n+1} \int_0^1 x^m (2-x)^n dx,
\end{equation*}
as $j \to \infty$. On the other hand, by Faulhaber's formula, we know that the sum is a polynomial in $j$. Thus, we have that

\begin{equation} \label{sum 4}
    \sum_{k=0}^{j-1} k^m (2j-k)^n=j^{m+n+1} \int_0^1 x^m (2-x)^n dx+O(j^{m+n}),
\end{equation}
and then using (\ref{sum 3}) and (\ref{sum 4}) in (\ref{sum 2}) yields

\begin{equation*}
    \sum_{f \in \mathcal{M}_j} \sum_{f^2=f_1 f_2} d(f_1)^m d(f_2)^n=\frac{q^j j^{m+n+2}}{2 \zeta_q(2)} \int_0^1 \left( x^{m+1} (2-x)^n+x^{n+1} (2-x)^m \right) dx+O(q^j j^{m+n+1}).
\end{equation*}
Therefore, by Faulhaber's formula,

\begin{align*}
    & \sum_{f \in \mathcal{M}_{\leq h/2}} \frac{\tau^{(m,n)}(f^2)}{|f|} \\
    & =(-\log q)^{m+n} \sum_{j=0}^{h/2} q^{-j} \sum_{f \in \mathcal{M}_j} \sum_{f^2=f_1 f_2} d(f_1)^m d(f_2)^n \\
    & =\frac{(-\log q)^{m+n}}{2 \zeta_q(2)} \int_0^1 \left( x^{m+1} (2-x)^n+x^{n+1} (2-x)^m \right) dx \sum_{j=0}^{h/2} j^{m+n+2}+O \left( \sum_{j=0}^{h/2} j^{m+n+1} \right) \\
    & =\frac{(-\log q)^{m+n}}{2^{m+n+4} (m+n+3) \zeta_q(2)} \cdot h^{m+n+3} \int_0^1 \left( x^{m+1} (2-x)^n+x^{n+1} (2-x)^m \right) dx+O(h^{m+n+2}).
\end{align*}
Recalling the definition of $A(m,n)$ and choosing $h \in \{2g,2g-1\}$ completes the proof.

\end{proof}

We are now ready to prove Theorem \ref{joint moment thm}.

\begin{proof}[Proof of Theorem \ref{joint moment thm}]
We begin with the approximate functional equation for the product of two shifted $L$-functions given in \cite[Lemma 2.1]{BFK23}. Namely, we have that

\begin{equation*}
    L(\tfrac{1}{2}+\alpha,\chi_P) L(\tfrac{1}{2}+\beta,\chi_P)=\sum_{f\in\mathcal{M}_{\leq 2g}} \frac{\tau_{\alpha,\beta}(f) \chi_P(f)}{|f|^{1/2}}+q^{-2g(\alpha+\beta)} \sum_{f\in\mathcal{M}_{\leq 2g-1}} \frac{\tau_{-\alpha,-\beta}(f) \chi_P(f)}{|f|^{1/2}}.
\end{equation*}
Using the approximate functional equation, we write

\begin{equation} \label{expression 1}
   \frac{1}{|\P|} \sum_{P \in \P} L(\tfrac{1}{2}+\alpha,\chi_P) L(\tfrac{1}{2}+\beta,\chi_P)=F_{2g}(\alpha,\beta)+q^{-2g(\alpha+\beta)} F_{2g-1}(-\alpha,-\beta),
\end{equation}
where

\begin{equation*}
    F_{2g}(\alpha,\beta):=\sum_{f\in\mathcal{M}_{\leq 2g}} \frac{\tau_{\alpha,\beta}(f)}{|f|^{1/2}} \frac{1}{|\P|} \sum_{P\in\P} \chi_P(f),
\end{equation*}
and $F_{2g-1}(\alpha,\beta)$ is given by a similar expression with $\mathcal{M}_{\leq 2g}$ replaced by $\mathcal{M}_{\leq 2g-1}$. For integers $m,n \geq 0$ and $h \in \{2g,2g-1\}$, we then have that

\begin{align*}
    \frac{\partial^{m+n}}{\partial\alpha^m \partial\beta^n} F_h(\alpha,\beta) |_{\alpha=\beta=0}=T_h(m,n)=\sum_{f\in\mathcal{M}_{\leq h}} \frac{\tau^{(m,n)}(f)}{|f|^{1/2}} \frac{1}{|\P|} \sum_{P\in\P} \chi_P(f).
\end{align*}
Thus, differentiating (\ref{expression 1}) with respect to $\alpha$ and $\beta$ leads to 

\begin{align*}
    & \frac{1}{|\P|} \sum_{P \in \P} L^{(\mu)}(\tfrac{1}{2},\chi_P) L^{(\nu)}(\tfrac{1}{2},\chi_P) \nonumber \\
    &\qquad =T_{2g}(\mu,\nu)+(-1)^{\mu+\nu} \sum_{m=0}^{\mu} \sum_{n=0}^{\nu} \binom{\mu}{m} \binom{\nu}{n} (2g \log q)^{\mu+\nu-m-n} \, T_{2g-1}(m,n).
\end{align*}
Theorem \ref{joint moment thm} then follows on applying the asymptotic formula for $T_h(m,n)$ in Lemma \ref{T asymptotic}.
\end{proof}

\vspace{0.5cm}

\noindent \textit{Acknowledgments.}
I would like to thank my supervisor Julio Andrade for suggesting this problem to me and for his guidance throughout. The research of the author was supported by an EPSRC Standard Research Studentship (DTP) at the University of Exeter.




\begin{thebibliography}{}


\bibitem{And19} \label{And19}
J. C. Andrade, {\it Mean values of derivatives of $L$-functions in function fields: III} Proc. Roy. Soc. Edinburgh Sect. A \textbf{149} (2019), 905--913.

\bibitem{AB24} \label{AB24}
J. C. Andrade and C. G. Best, {\it Joint moments of derivatives of characteristic polynomials of random symplectic and orthogonal matrices}, J. Phys. A: Math. Theor. \textbf{57} (2024), no. 20, 205205.

\bibitem{AJ21} \label{AJ21}
J. C. Andrade and H. Jung, {\it Mean values of derivatives of $L$-functions in function fields: IV}, J. Korean Math. Soc. \textbf{58} (2021), no. 6, 1529--1547.

\bibitem{AJS21} \label{AJS21}
J. C. Andrade, H. Jung and A. Shamesaldeen, {\it The integral moments and ratios of quadratic Dirichlet $L$–functions over monic irreducible polynomials in $F_q[t]$}, Ramanujan J. \textbf{56} (2021), 23--66.

\bibitem{AK12} \label{AK12}
J. C. Andrade and J. P. Keating, {\it The mean value of $L(\tfrac{1}{2},\chi)$ in the hyperelliptic ensemble}, J. Number Theory \textbf{132} (2012), 2793--2816.


\bibitem{AR16} \label{AR16}
J. Andrade and S. Rajagopal, {\it Mean values of derivatives of $L$-functions in function fields: I}, J. Math. Anal. Appl. \textbf{443} (2016), 526--541.

\bibitem{AY21} \label{AY21}
J. C. Andrade and M. Yiasemides, {\it The fourth moment of derivatives of Dirichlet $L$-functions in function fields}, Math. Z. \textbf{299} (2021), 671--697.

\bibitem{BJ19} \label{BJ19}
S. Bae and H. Jung, {\it Note on the mean values of derivatives of quadratic Dirichlet $L$-functions in function fields}, Finite Fields Appl. \textbf{57} (2019), 249--267.


\bibitem{BFK23} \label{BFK23}
H. M. Bui, A. Florea and J. P. Keating, {\it The Ratios Conjecture and upper bounds for negative moments of $L$-functions over function fields}, Trans. Amer. Math. Soc. \textbf{376} (2023), 4453--4510.

\bibitem{Conrey88} \label{Conrey88}
J. B. Conrey, {\it The fourth moment of derivatives of the Riemann zeta-function}, Quart. J. Math. \textbf{39} (1988), 21--36.

\bibitem{CFKRS05} \label{CFKRS05}
J. B. Conrey, D .W. Farmer, J. P. Keating, M. O. Rubinstein and N. C. Snaith, {\it Integral moments of $L$-functions}, Proc. London Math. Soc. \textbf{91} (2005), 33--104.

\bibitem{CRS06} \label{CRS06}
J. B. Conrey, M. O. Rubinstein and N. C. Snaith, {\it Moments of the derivative of characteristic polynomials with an application to the Riemann zeta function}, Commun. Math. Phys. \textbf{267} (2006), 611--629.

\bibitem{DD21} \label{DD21}
G. Djankovi{\'c} and D. \DJ oki{\'c}, {\it The mixed second moment of quadratic Dirichlet $L$-functions over function fields}, Rocky Mountain J. Math. \textbf{51} (2021), no. 6, 2003--2017.

\bibitem{Florea17a} \label{Florea17a}
A. Florea, {\it Improving the error term in the mean value of $L(1/2,\chi)$ in the hyperelliptic ensemble}, Int. Math. Res. Not. \textbf{20} (2017), 6119--6148.

\bibitem{Florea17b} \label{Florea17b}
A. Florea, {\it The second and third moment of $L(1/2,\chi)$ in the hyperelliptic ensemble}, Forum Math. \textbf{29} (2017), no. 4, 873--892.

\bibitem{Florea17c} \label{Florea17c}
A. Florea, {\it The fourth moment of quadratic Dirichlet $L$-functions over function fields}, Geom. Funct. Anal. \textbf{27} (2017), no. 3, 541--595.

\bibitem{Ingham27} \label{Ingham27}
A. E. Ingham {\it Mean-value theorems in the theory of the Riemann zeta-function}, Proc. London Math. Soc. (2) \textbf{27} (1927), no. 4, 273--300.

\bibitem{Jung22} \label{Jung22}
H. Jung, {\it Mean values of derivatives of quadratic prime Dirichlet $L$-functions in function fields}, Commun. Korean Math. Soc. \textbf{37} (2022), no. 3, 635--648.

\bibitem{KS00a} \label{KS00a}
J. P. Keating and N. C. Snaith, {\it Random Matrix Theory and $\zeta(1/2+it)$}, Commun. Math. Phys. \textbf{214} (2000), 57--89.


\bibitem{Ros02} \label{Ros02}
M. Rosen, {\it Number theory in function fields}, Graduate Texts in Mathematics, vol. 210, Springer-Verlag, New York, 2002.

\bibitem{Rud10} \label{Rud10}
Z. Rudnick, {\it Traces of high powers of the Frobenius class in the hyperelliptic ensemble}, Acta Arith. \textbf{143} (2010), no. 1, 81--99.

\bibitem{Weil48} \label{Weil48}
A. Weil, {\it Sur les courbes alg\'ebriques et les vari\'et\'es qui s'en d\'eduisent}, Actualit\'es Sci. Ind., no.
1041, Publ. Inst. Math. Univ. Strasbourg \textbf{7}, Hermann et Cie., Paris (1948).

\end{thebibliography}
\end{document}